\documentclass[runningheads,12pt]{article}

\usepackage{amsmath}
\usepackage{amsfonts}

\newtheorem{theorem}{Theorem}

\newtheorem{definition}[theorem]{Definition}
\newtheorem{example}[theorem]{Example}

\newtheorem{proposition}[theorem]{Proposition}
\newtheorem{remark}[theorem]{Remark}

\newenvironment{proof}[1][Proof]{\noindent \textbf{#1.} }{\  \rule{0.5em}{0.5em}}

\begin{document}
\title{\mbox{The paradox of the infinity}}

\author{MOHAMED AYAD and OMAR KIHEL}

\date{}

\maketitle

\pagenumbering{arabic}
\section{Introduction}

Let $(\bf P)$ be a property defined on the set of the natural numbers and $A_{n}$ be the set of positive integers less or equal $n$, where $n$ is a positive integer. If an integer is selected at random from $A_{n}$, then the probability $\sigma_{n}$ that it satisfies $(\bf P)$ is the number of elements in $A_{n}$ that satisfy $(\bf P)$ divided by $n$. If the limit of $\sigma_{n}$ when $n$ tends to infinity exists, this limit is called the natural density or asymptotic density of the integers satisfying $(\bf P)$. Intuitively, this density measures how frequently an integer satisfies $(\bf P)$. One can define analogously the natural density on any infinite subset of the set of natural numbers. Let's consider the following situation:

\textit{Let $E$ be an infinite set on which a property $(\bf P)$ is defined. Suppose that $E=\cup_{i\in I} E_i$ is a partition, where each $E_i$ is infinite. Suppose also that, in each $E_i$, the number of elements satisfying $(\bf P)$ is finite. Then, clearly the density of the elements satisfying $(\bf P)$ is 0 in every $E_i$. Is it possible that the density of the subset of $E$ containing all the elements satisfying $(\bf P)$ will be at least equal to $ 1/2$?}

We were first confronted with this situation while reading the paper of Arno et al. [1].  In fact, it is in the paper [1] where it is shown that the density of certain algebraic numbers in $\overline{\mathbb{Q}}$, which we will call Arno et al. numbers in section 5, is equal to $1/\zeta(3)$.  We have partitioned $\overline{\mathbb{Q}}$ in a way that suggests these Arno et al. numbers are rare.  This phenomenom struck us as contradictory, which lead us to consider the situation in greater detail. We will show in the sequel, through two examples, that the answer to the above question may be positive. At first glance, this problem resembles to the so called Simpson paradox in probability and statistics. In this paper, when we say the density, we mean the natural density.

 \section{First example}
Let  $E=\mathbb{N}\setminus\{0,1\}$. In this set the relation  $(\bf P)$ will be : $x$ is even. For any  $i\geq 1$, let
\begin{equation*}
E_i= \{ 2i\}\cup \{2^ik+1,k\geq 1\text{ odd}\}
\end{equation*}
The  $E_i$ constitute a partition of $E$; each class contains a unique element satisfying $(\bf P)$. On the other hand, the density in $E$ of the elements satisfying $(\bf P)$ is equal to $1/2$.\\

Before producing the second example, we need to state, in the following section, some results on the denominator and constant coefficient of an algebraic number. The reader who is familiar with these notions may skip this section.
\section{Denominator of an algebraic number}
Let $\gamma$ be an algebraic number of degree $n$, \begin{equation*}g(x)=c_nx^n+c_{n-1}x^{n-1}+\ldots+c_1x+c_0\end{equation*} be the unique  irreducible polynomial with integral coefficients such that $g(\gamma)=0$ and $c_n>0$. This polynomial will be called the minimal polynomial of $\gamma$ over $\mathbb{Z}$. Let
\begin{equation} I=I(\gamma)=\{k\in\mathbb{Z},\,\, k\gamma\,\,\mbox{an algebraic integer}\}
\end{equation}
Then $I(\gamma)$ is a principal ideal of $\mathbb{Z}$. Since $c_n\gamma$ is an algebraic integer, then $c_n\in I(\gamma)$. Hence $I(\gamma)\neq(0)$. Let $d>0$ be a generator of $I(\gamma)$, then $d$ is the smallest positive integer such that $d\gamma$ is integral. Moreover $d\mid c_n$. Call this integer $d$ the denominator of $\gamma$. Let $\theta=d\gamma$, then $\theta$ is an algebraic integer, called the numerator of $\gamma$. The leading coefficient  $c_n$ of $g(x)$ will be called the leading coefficient of $\gamma$. Denote by $d(\gamma)$ and $c(\gamma)$ the denominator and the leading coefficient of $\gamma$ respectively.

\begin{proposition}
Let $\gamma$ be an algebraic number  of degree $n$, $\gamma_1,\ldots,\gamma_n$ be the  list of its conjugates over $\mathbb{Q}$. For any $i=1,\ldots,n$, let $s_i(\overrightarrow{\gamma})=s_i(\gamma_1,\ldots,\gamma_n)$ be the elementary symmetric function of degree $i$ of the $\gamma_j$. Let
\begin{equation}J=J(\gamma)=\{k\in\mathbb{Z},\,\,ks_i(\vec{\gamma})\in\mathbb{Z}\,\mbox{for}\,\,i=1,\ldots,n\},
\end{equation}
then $J$ is a principal ideal of  $\mathbb{Z}$ generated by $c(\gamma)$. Moreover $c(\gamma)\mid d^n$.
\end{proposition}
\begin{proof} Let $c^{'}$ be a  generator of $J(\gamma)$ and let $d=d(\gamma)$. Set $\gamma=\theta/d$ where $\theta$ is an algebraic integer. Since
\begin{equation*}
d^ns_k(\gamma)=d^n\sum\gamma_{i_1}\cdots\gamma_{i_k}=d^n\sum(\theta_{i_1}/d)\cdots(\theta_{i_k}/d)
\end{equation*}
then $d^n\in J(\gamma)$ thus $c^{'} \mid d^n$.
Let $c=c(\gamma)$ and let $g(x)$ be the minimal polynomial of
 $\gamma$ over $\mathbb{Z}$, then
\begin{align*} g(x)&=c(x-\gamma_1)\cdots(x-\gamma_n)\\
&=c\big(x^n-s_1(\gamma)x^{n-1}+\ldots+(-1)^ns_n(\gamma)\big)\\
&=cx^n-cs_1(\gamma)x^{n-1}+\ldots+(-1)^ncs_n(\gamma).\end{align*}
Since $g(x)\in\mathbb{Z}[x]$ then $cs_k(\vec{\gamma}) \in\mathbb{Z}$ for $k=1,\ldots,n$, hence $c\in J(\gamma)$, which implies $c^{'}\mid c$.
On the other hand let
\begin{align*} g_1(x)&=c^{'}(x-\gamma_1)\cdots(x-\gamma_n)\\
&=c^{'}x^n-c^{'}s_1(\gamma)x^{n-1}+\ldots+(-1)^nc^{'}s_n(\gamma)\\
&\in\mathbb{Z}[x].
\end{align*}
Since $g_1(\gamma)=0$, then $g(x)$ divides $g_1(x)$ in $\mathbb{Q}[x]$. Set $g_1(x)= g(x)q$ with $q\in\mathbb{Q}$. Using the content of polynomials, denoted by cont, We have $cont(g_1)=qcont(g)=q$, hence $q\in\mathbb{Z}$. The comparison of the leading coefficients leads to $c^{'}=cq$, hence $c\mid c^{'}$. Finally we get $c=c^{'}$.
\end{proof}
\begin{remark} We have shown that $d(\gamma)\mid c(\gamma)$ and $c(\gamma)\mid d(\gamma)^n$, hence $d(\gamma)$ and $c(\gamma)$ have the same prime factors.
\end{remark}

\begin {example} In the following example we compare the  denominator and the leading coefficient. In particular we look for the condition $d(\gamma)=c(\gamma)$.

Let $\gamma=\sqrt{2}/(q2^k)$ with  $q$ positive and odd and $k\geq 0$. Here $d(\gamma)=q2^k$. $\gamma$ is a root of the equation $2^{2k}q^2\gamma^2-2=0$ which is not always irreducible over $\mathbb{Z}$. In the different cases which follow we  precise the minimal polynomial $g(x)$ of $\gamma$ over $\mathbb{Z}$.

$\bullet$ $k=0$, $g(x)=q^2x^2-2$, hence $d(\gamma)=q$ and $c(\gamma)=q^2>d(\gamma)$ except if $q=1$, i.e. $\gamma=\sqrt{2}$.

$\bullet$ $k=1$ and $q=1$. Here  $g(x)=2x^2-1$, $d(\gamma)=c(\gamma)=2$.

$\bullet$ $k=1$ and $q\geq 2$. Here  $g(x)=2q^2x^2-1$, hence $d(\gamma)=2q$ and $c(\gamma)=2q^2>d(\gamma)$.

$\bullet$ $k\geq 2$ and $q\geq 1$. Here  $g(x)=2^{2k-1}q^2x^2-1$, hence $d(\gamma)=2^kq$ and $c(\gamma)=2^{2k-1}q^2>d(\gamma)$.

 We thus get $c(\gamma)=d(\gamma)$ if and only if $\gamma=\sqrt{2}$ or $\gamma=\sqrt{2}/2$. This example suggests that we almost always  have $c(\gamma)>d(\gamma)$.
\end{example}
\section{Result of Arno-Robinson-Wheeler}
These three authors have shown that the density in $\overline{\mathbb{Q}}$ of the algebraic numbers $\gamma$ such that $c(\gamma)=d(\gamma)$ is equal to $1/\zeta(3)=0.8319...$. These numbers will be called Arno et al. numbers.
The height of a polynomial with integral coefficients, $P(x)=a_kx^k+\ldots+a_0$, is defined by $H(P)=\max_i|a_i|$. The height of an algebraic number is the height of its minimal polynomial over $\mathbb{Z}$. To prove their result, Arno et al. define the sets \begin{equation*}{\cal{A}}_d(H)=\{\gamma\in\overline{\mathbb{Q}}, deg(\gamma)=d,\,\,\mbox{and}\,\, H(\gamma)\leq H\}\,\, \mbox{and}\end{equation*} \begin{equation*}{\hat{\cal{A}}}_d(H)=\{\gamma\in{\cal{A}}_{d}(H)\}, c(\gamma)=d(\gamma)\},\end{equation*}
 where $H$ is a positive integer and $H(\gamma)$ denotes the height of $\gamma$. After obtaining an asymptotic formula of
$|{\hat{\cal{A}}}_d(H)|/|{\cal{A}}_d(H)|$, they deduce that $$\lim_{d\to\infty}\lim_{H\to \infty}|\cup_{k\leq d}{\hat{\cal{A}}}_k(H)|/|\cup_{k\leq d}{\cal{A}}_k(H)|=1/\zeta(3).$$
This result makes appear that the Arno et al. numbers are very frequent in the set of the algebraic numbers. In the sequel we define a partition of $\overline{\mathbb{Q}}$ which suggests that the Arno et al. numbers are rare. With the characterizations of $d(\gamma)$ and $c(\gamma)$ given above this result of Arno et al. seems surprising. Actually $d(\gamma)$ is the smallest positive integer $d$ such that $d\gamma$ is an algebraic integer. In contrast,
$c(\gamma)$ is the smallest positive integer $c$ such that all the products \begin{equation*}c(\gamma_1+\ldots+\gamma_n), c(\gamma_1\gamma_2+\ldots+\gamma_{n-1}\gamma_n),\ldots,c\gamma_1\cdots\gamma_n
\end{equation*}
are rational integers. For the first product, we clearly see that  $c$ may be replaced by $d$, but for the others? Arno's et al. result tells us that it is also true with a high probability.

\section{Second example}
\begin
{definition}  Let $K$ be a number field of degree $n$ over $\mathbb{Q}$, $\theta$ be an algebraic integer of $K$, and $ f(x)=x^n+a_{n-1}x^{n-1}+\ldots+a_1x+a_0$ the characteristic polynomial of $\theta$ over $\mathbb{Q}$. The norm  of  $\theta$ over $\mathbb{Q}$, denoted by $N_{\mathbb{Q}(\theta)/\mathbb{Q}}(\theta)$, is defined by $N_{\mathbb{Q}(\theta)/\mathbb{Q}}(\theta)= (-1)^{n}a_{0}$.
\end{definition}
\begin{proposition} Let $\gamma=\theta/d$ be an algebraic number of degree $n\geq 2$, where $\theta$ is an algebraic integer and $d=d(\gamma)$. Suppose that $\gamma$ is an Arno et al. number, then $d^{n-1}\mid N_{\mathbb{Q}(\theta)/\mathbb{Q}}(\theta)$.
\end{proposition}

\begin{proof} Let
\begin{equation*} f(x)=x^n+a_{n-1}x^{n-1}+\ldots+a_1x+a_0
\end{equation*}
be the minimal polynomial of $\theta$ over $\mathbb{Q}$. Since $f(d\gamma)=0$, we have \begin{equation*}
d\gamma^n+a_{n-1}\gamma^{n-1}+\cdots+(a_1/d^{n-2})\gamma+(a_0/d^{n-1})=0.
\end{equation*}
The minimal polynomial of $\gamma$ over $\mathbb{Z}$ has the form
\begin{equation*}
g(x)=dx^n+c_{n-1}x^{n-1}+\ldots+c_1x+c_0.
\end{equation*}
It follows that $g(x)$ divides the polynomial \begin{equation*}
dx^n+a_{n-1}x^{n-1}+\cdots+(a_1/d^{n-2})x+(a_0/d^{n-1}).
\end{equation*}
in $\mathbb{Q}[x]$. Let $q\in\mathbb{Q}$ such that \begin{equation*}
dx^n+a_{n-1}x^{n-1}+\cdots+(a_1/d^{n-2})x+(a_0/d^{n-1})=qg(x).
\end{equation*}
Identifying the leading coefficients, we get $q=1$ and since $c_0$ is an integer, then $d^{n-1}\mid a_0$.
\end{proof}

Make a partition of $\overline{\mathbb{Q}}\setminus \mathbb{Q} $ by putting in the same  class, say $C(\theta)$, all the algebraic numbers whose numerator is equal to the algebraic integer $\theta$. The preceding proposition  shows that any class $C(\theta)$ contains a finite number (may be  $0$) of Arno et al. numbers. We thus have a second example of a set, namely $\overline{\mathbb{Q}}\setminus \mathbb{Q} $, and the property $(\bf P)$ states: $\gamma$ is an Arno et al. number.

We have excluded from $\overline{\mathbb{Q}}$ the rational numbers because they are Arno et al. numbers and any corresponding class $C(\theta)$ is infinite.

One may think that the Arno et al. numbers are mostly frequent when their degrees  are large enough. But, let $\gamma=\theta/d$ be an Arno et al. algebraic number of degree $n$. Let $q$ be a positive integer divisible by some prime number $p$ such that $p\not\mid  N_{\mathbb{Q}(\theta)/\mathbb{Q}}(\theta)$, then by the preceding proposition the number $\gamma_q=\theta/q$ has the same degree as $\gamma$ and generates the same field, but it is not an Arno et al. number. Obviously the number of $\gamma_q$'s is infinite.

\section{Denominator and leading coefficient in the formal case }

\begin{proposition} Let $t_0,t_1,\ldots,t_n,x$ be algebraically independent variables over $\mathbb{Z}$ and let
\begin{equation*}
G(x)=t_nx^n+\ldots+t_1x+t_0\in \mathbb{Z}[t_0,t_1,\ldots,t_n][x].
\end{equation*}
Set $\vec t=(t_0,t_1,\ldots,t_n)$. Let $\Gamma$ be a root of $G(x)$ in an algebraic closure of $\mathbb{Q}(\vec t)$.
Let \begin{equation*}
I=\{D(\vec t)\in \mathbb{Z}[\vec t],\,\, D\Gamma \,\,\mbox{is integral over}\,\, \mathbb{Z}[\vec t]   \},
\end{equation*}
then $I$ is a principal ideal of $\mathbb{Z}[\vec t]$ generated by $t_n$.
\end{proposition}

\begin{proof} Obviously $t_n\in I$. Let $D(\vec t)\in I$, then since  \begin{equation*} t_n\Gamma^n+t_{n-1}\Gamma^{n-1}+\ldots+t_1\Gamma+t_0=0,
\end{equation*}
we have
\begin{equation*}
\big(D(\vec t)\Gamma\big)^n+\frac {t_{n-1}D(\vec t)}{t_n}\big(D(\vec t)\Gamma\big)^{n-1}+\ldots+\frac {t_1D(\vec t)^{n-1}}{t_n}\big(D(\vec t)\Gamma\big)+\frac {t_0D(\vec t)}{t_n}=0,
\end{equation*}
hence $t_n\mid D(\vec t)$ thus $I=t_n\mathbb{Z}[\vec t]$.
\end{proof}

This proposition shows that if $\Gamma$ is a root of the generic polynomial, then its denominator and leading coefficient are equal. Since the minimal polynomial over $\mathbb{Z}$ of any algebraic number may be obtained by specializing the coefficients of the generic polynomial, it is tempting to conclude that this confirms the result of Arno et al.

\textbf{Question:} Is it  possible to find an uncountable set $E$ and a property $(\bf P)$ satisfying the conditions stated in the introduction? Our feeling is that the answer is negative.


\author{Mohamed Ayad\\
	Laboratoire de Math{\'e}matiques Pures et Appliqu{\'e}es\\
	Universit\'e du Littoral, F-62228 Calais, France\\
	ayadmohamed502@yahoo.com\\
	Omar~Kihel\\ Department of Mathematics\\ Brock University, Ontario, Canada L2S 3A1\\okihel@brocku.ca}


\begin{thebibliography}{99}
\bibitem{A.R.W} S. Arno, M. L. Robinson, F. S. Wheeler, On Denominators of Algebraic Numbers and Integer Polynomials,
Journal of Number Theory 57, (1996), 292-302.
\end{thebibliography}
\end{document}